\documentclass[a4paper,10pt]{amsart}
\usepackage{amsmath,amsthm,amssymb,latexsym,enumerate,color,hyperref}
\usepackage{graphicx}
\usepackage{color}
\usepackage{eucal}
\numberwithin{equation}{section}
\begin{document}

\newtheorem{thm}{Theorem}[section]
\newtheorem{prop}[thm]{Proposition}
\newtheorem{lem}[thm]{Lemma}
\newtheorem{cor}[thm]{Corollary}
\newtheorem{ex}[thm]{Example}                         
\newtheorem{rem}[thm]{Remark}
\newtheorem*{defn}{Definition}
\newtheorem*{note}{Note}

\newcommand{\DD}{\mathbb{D}}
\newcommand{\NN}{\mathbb{N}}
\newcommand{\ZZ}{\mathbb{Z}}
\newcommand{\QQ}{\mathbb{Q}}
\newcommand{\RR}{\mathbb{R}}
\newcommand{\CC}{\mathbb{C}}
\renewcommand{\SS}{\mathbb{S}}

\renewcommand{\theequation}{\arabic{section}.\arabic{equation}}

\newcommand{\supp}{\mathop{\mathrm{supp}}}    

\newcommand{\re}{\mathop{\mathrm{Re}}}   
\newcommand{\im}{\mathop{\mathrm{Im}}}   
\newcommand{\dist}{\mathop{\mathrm{dist}}}  

\newcommand{\spn}{\mathop{\mathrm{span}}}   
\newcommand{\ind}{\mathop{\mathrm{ind}}}   
\newcommand{\rank}{\mathop{\mathrm{rank}}}   
\newcommand{\Fix}{\mathop{\mathrm{Fix}}}   
\newcommand{\codim}{\mathop{\mathrm{codim}}}   
\newcommand{\conv}{\mathop{\mathrm{conv}}}   

\newcommand{\pa}{\partial}
\newcommand{\ve}{\varepsilon}
\newcommand{\zi}{\zeta}
\newcommand{\Si}{\Sigma}
\newcommand{\cA}{{\mathcal A}}
\newcommand{\cG}{{\mathcal G}}
\newcommand{\cH}{{\mathcal H}}
\newcommand{\cI}{{\mathcal I}}
\newcommand{\cJ}{{\mathcal J}}
\newcommand{\cK}{{\mathcal K}}
\newcommand{\cL}{{\mathcal L}}
\newcommand{\cN}{{\mathcal N}}
\newcommand{\cR}{{\mathcal R}}
\newcommand{\cS}{{\mathcal S}}
\newcommand{\cT}{{\mathcal T}}
\newcommand{\cU}{{\mathcal U}}
\newcommand{\cC}{\mathcal{C}}
\newcommand{\cM}{\mathcal{M}}
\newcommand{\De}{\Delta}
\newcommand{\cX}{\mathcal{X}}
\newcommand{\cP}{\mathcal{P}}

\newcommand{\ol}{\overline}
\newcommand{\ul}{\underline}

\newcommand{\AC}{\mathop{\mathrm{AC}}}   
\newcommand{\Lip}{\mathop{\mathrm{Lip}}}   
\newcommand{\es}{\mathop{\mathrm{esssup}}}   
\newcommand{\les}{\mathop{\mathrm{les}}}   

\newcommand{\la}{\lambda}
\newcommand{\La}{\Lambda}    
\newcommand{\de}{\delta}    
\newcommand{\fhi}{\varphi} 
\newcommand{\vro}{\varrho} 
\newcommand{\ga}{\gamma}    
\newcommand{\ka}{\kappa}   

\newcommand{\core}{\heartsuit}
\newcommand{\diam}{\mathrm{diam}}

\newcommand{\lan}{\langle}
\newcommand{\ran}{\rangle}
\newcommand{\tr}{\mathop{\mathrm{tr}}}
\newcommand{\diag}{\mathop{\mathrm{diag}}}
\newcommand{\dv}{\mathop{\mathrm{div}}}

\newcommand{\al}{\alpha}
\newcommand{\be}{\beta}
\newcommand{\Om}{\Omega}
\newcommand{\na}{\nabla}

\newcommand{\OO}{\mathcal{O}}

\newcommand{\nr}{\Vert}
\newcommand{\weak} {\rightharpoonup}
 
\newcommand{\om}{\omega}
\newcommand{\si}{\sigma}
\newcommand{\te}{\theta}
\newcommand{\Ga}{\Gamma}

\newcommand{\ds}{\displaystyle}

\title{Explicit complex-valued solutions \\ of the 2D eikonal equation}

\author{Rolando Magnanini} 
\address{Dipartimento di Matematica ed Informatica ``U.~Dini'',
Universit\` a di Firenze, viale Morgagni 67/A, 50134 Firenze, Italy.}
   %\curraddr{...}
    \email{rolando.magnanini@unifi.it}
    \urladdr{http://web.math.unifi.it/users/magnanin}

\dedicatory{To Bob Gilbert, a friend and a mentor, with profound gratitude. \\
He is the man who introduced me to the deep beauties of the complex variable.}

\begin{abstract}
We present a method to obtain explicit solutions of the com\-plex eikonal equation in the plane. This equation arises in the ap\-prox\-i\-ma\-tion of Helmholtz equation by the WKBJ or EWT methods. We obtain the complex-valued solutions (called eikonals) as parametrizations in a complex variable. We consider both the cases of constant and non-constant index of refraction. In both cases, the relevant parametrizations depend on some holomorphic function. In the case of non-constant index of refraction, the parametrization also depends on some extra exponential complex-valued function and on a quasi-conformal homeomorphism. This is due to the use of the theory of pseudo-analytic functions and the related similarity principle. The parametrizations give information about the formation of caustics and the light and shadow regions for the relevant eikonals.
\end{abstract}

\keywords{Complex eikonal equation, geometrical optics, explicit solutions, pseudo-analytic functions}
    \subjclass[2010]{78A05, 35F20, 35C05, 35A22, 35A30, 30G20}

\maketitle

\pagestyle{myheadings} \thispagestyle{plain} \markboth{R. MAGNANINI}{SOLUTIONS OF THE COMPLEX EIKONAL EQUATION}

\section{Introduction} \label{section introd}
Let $n$ be some continuous {\it strictly positive real-valued} function of $x\in\RR^N$ and consider the following nonlinear first-order differential equation
\begin{equation} 
\label{eikonal-equation}
\sum_{j=1}^N \Big( \frac{\pa \phi}{\pa x_j} \Big)^2 = n^2;
\end{equation}
\eqref{eikonal-equation} is the so-called {\it eikonal equation} that arises in many a field of science, 
such as  optics, acoustics, nuclear physics, control theory, to name a few.

An important motivation to study \eqref{eikonal-equation} is that its solutions help to understand the asymptotic behaviour of solutions of the {\it Helmholtz equation}
(that ensues from Maxwell's system),
\begin{equation} \label{eq helmholtz}
\Delta W + k^2 n^2 W =0,
\end{equation}
for large values of the {\it wave number} $k.$ 
Here, $n$ has the meaning of an {\it index of refraction}, whose reciprocal is proportional 
to the velocity of propagation of waves through the medium.

An expansion, originating from the so-called WKBJ method (after Wentzel, Kramers, Brillouin, and Jeffreys) represents solutions of \eqref{eq helmholtz}, asymptotically as $k\to+\infty$, by
\begin{equation} 
\label{sviluppo WKBJ}
W \simeq e^{ik\phi} \sum_{m=0}^\infty A_m (ik)^{-m}.
\end{equation}
Here, $\phi$ is a real-valued {\it phase function}, which is a solution of \eqref{eikonal-equation}, and the $A_m$'s are solutions of the so-called {\it transport equations}:
%\begin{subequations}
\begin{eqnarray*}
&& 2 \na \phi \cdot \na A_0 + \De \phi\, A_0 = 0, \\
&& 2 \na \phi \cdot \na A_m + \De \phi\, A_m = \De A_{m-1}, \quad (m=1,2,\ldots);
\end{eqnarray*}
%\end{subequations}
see e.g. \cite{LK}.
\par
The expansion \eqref{sviluppo WKBJ} is successful in accounting for those phenomena pertaining {\it geometrical optics}, such as the insurgence of {\it caustics} via the mechanism of rays. However,  the expansion cannot describe those phenomena, such as the development of {\it evanescent waves}, that take place past a caustic and create a somewhat blurred {\it shadow region}.
\par
More information can be derived by the theory of {\it Evanescent Wave Tracking (EWT)} proposed by Felsen and coworkers
(see \cite{BM}-\cite{HF}). In EWT, \eqref{sviluppo WKBJ} is replaced by the following expansion
\begin{equation} 
\label{sviluppo EWT}
W= e^{\la + kv} \Big\{ \cos(\mu+ku) + \OO\Big( \frac{1}{k} \Big) \Big\} \quad \textmd{as } k\to \infty.
\end{equation}
The real-valued functions $u,v,\la,\mu$ are all independent of $k$ and turn out to obey both
\begin{equation*} 
\label{eq uv sist 1}
|\na u|^2 - |\na v|^2 - n^2 = 0, \quad \na u \cdot \na v = 0,
\end{equation*}
and
\begin{equation*} 
\label{eq uv sist 2}
\na u \cdot \na \la + \na v \cdot \na \mu + \frac{1}{2} \Delta u = 0, \ \  \quad
\na v \cdot \na \la - \na u \cdot \na \mu + \frac{1}{2} \Delta v = 0.
\end{equation*}
Notice that the former system for $u$ and $v$ tells us that $\phi=u+iv$ is a complex-valued solution of \eqref{eikonal-equation}.
\par
Observe that in the region in which $v\equiv 0$, $\phi$ is real-valued and the character of the expansion in \eqref{sviluppo EWT} is purely oscillatory. This corresponds to the regime of geometrical optics (the light region) in which light propagates along rays. When $v<0$ instead, an inspection of \eqref{sviluppo EWT} tells us that $W$ decays exponentially for large values of $k$. This creates the evanescent waves, which vanish exponentially in the shadow region near the caustic.
\par
One more sophisticated asymptotic expansion, which powerfully helps to represent solutions of \eqref{eq helmholtz} on both sides of a caustic, and once more shows the utility of complex-valued solutions of \eqref{eikonal-equation}, is developed in a theory of Kravtsov and Ludwig (see \cite{Kr1}, \cite{Kr2},\cite{LBL},\cite{Lu},\cite{Lu1},\cite{Lu2},\cite{BM}). It reads as
\begin{equation*} 
\label{sviluppo Krav e Ludw}
W \simeq {\rm Ai} (k^{\frac{2}{3}} v) e^{ikU} \sum_{m=0}^\infty A_m (ik)^{-m} + i k^{-\frac{1}{3}} {\rm Ai}'(k^{\frac{2}{3}} v) e^{ikU} \sum_{m=0}^\infty B_m (ik)^{-m}, 
\end{equation*}
as $k\to \infty$. The real-valued functions $U$ and $V$ and the complex-valued ones $A_m$ and $B_m$ are independent on $k$; $Ai(\cdot)$ denotes the {\it Airy function} defined by the problem
\begin{equation*}
\label{eq Airy}
 {\rm Ai}'' (s)-s\,{\rm Ai}(s)= 0,\quad {\rm Ai}(0)= 3^{-\frac{2}{3}}\Ga(2/3)^{-1}, \quad {\rm Ai}'(0)= - 3^{-\frac{1}{3}}\Ga(1/3)^{-1}.
\end{equation*}
Here, $\Ga$ is Euler's gamma function. By Kravtsov and Ludwig's expansion and \eqref{eq helmholtz}, one gets that
\begin{equation*}
\label{eq U e V 10}
| \na U|^2 - V | \na V|^2 - n^2 =0,\quad \na U \cdot \na V =0.
\end{equation*}
Again observe that the complex-valued function $U+i\sqrt{V}$ satisfies \eqref{eikonal-equation}.

\medskip

Complex-valued solutions of \eqref{eikonal-equation} can be obtained by means of the method of {\it complex rays} (see \cite{CLOT}, \cite{MT1}). As an illustration, we carry out here the simplest case in which $N=2$ and $n\equiv 1$. Given a pair $(a,b)$, consider the parametrization
\begin{equation}
\label{rays}
x=a+\phi\,\cos\ga,\quad y = b+\phi\, \sin\ga.
\end{equation}
When $a,b\in\RR$, \eqref{rays} defines a real-valued solution $\phi$ of \eqref{eikonal-equation} as 
the distance of the point $(x,y)$ to the point $(a,b)$. If either $a$ or $b$ is not real, 
then a complex-valued solution $\phi$ of \eqref{eikonal-equation} can be constructed via \eqref{rays} 
by making complex the parameter $\ga$, but by keeping $x$ and $y$ real-valued.
For instance, for $a=0$ and $b=i,$ we obtain:
\begin{equation}
\label{dist}
\phi=\sqrt{x^2+(y-i)^2}=\sqrt{(z+1)(\ol{z}-1)}.
\end{equation}
Unfortunately, this approach appears to be difficult to use when the refractive index $n$ is non-constant.
\par
In this paper, we shall present another way to construct explicit complex-valued solutions of \eqref{eikonal-equation} in the plane $(x,y)$. The method can be extended to the case of non-constant index of refraction. The solutions are obtained as parametrizations of the following type
$$
\CC\supset A\ni\zi\mapsto (z(\zi),\phi(\zi))\in\CC^2.
$$ 
In fact, if we introduce the complex variable $z=x+i\,y$ and its conjugate $\ol{z}=x-i\,y$, \eqref{eikonal-equation} can be rewritten as:
\begin{equation}
\label{eikonal-equation in z zbar}
4\, \phi_z \phi_{\ol{z}} = n(z)^2.
\end{equation}
We then observe that, for fixed $z$, the point $(\phi_{z},\phi_{\ol{z}})$ belongs to a hyperbola (in $\CC^2$) that, for instance, we can parameterize by a function $\zi = \zi(z,\ol{z})$ as
\begin{equation*}
\phi_z = \frac{n}{2\zi}, \quad \phi_{\ol{z}} = \frac{n}{2}\, \zi.
\end{equation*}
This assumption gives a differential constraint on the parameter function $\zi(z,\ol{z})$, which  can be converted in an equation for its inverse $z=z(\zi,\ol{\zi})$. This can be solved and, by rather simple manipulations, we obtain solutions of \eqref{eikonal-equation in z zbar}  parametrized as
\begin{equation*}
z=z(\zi,\ol{\zi}), \quad \phi=\phi(\zi,\ol{\zi}).
\end{equation*}
\par
We begin with Section \ref{sec:Legendre-constant}, in which we explain in detail the method, when $n$ is constant. As a result, we obtain the (local) explicit parametrization of $\phi$:
\begin{equation}
\label{parametrization-constant}
\begin{array}{l}
\displaystyle z= \frac{f(\zi) + \zi^2 \overline{f(\zi)}}{1-|\zi|^4}, \\
\displaystyle  \phi= \frac{1}{1-|\zi|^4} \Big\{ \zi \ol{f(\zi)} + \frac{f(\zi)}{\zi} \Big\} - \frac{f(\zi)}{2\zi} + \frac{1}{2} \int \frac{f(\zi)}{\zi^2} d\zi.
\end{array}
\end{equation}
Here --- and this is remarkable --- $f(\zi)$ is some analytic function of the variable $\zi$.
\par
In the case in which $n$ is non-constant, the situation is more complicated, but perfectly consistent. We present this case in Section \ref{sec:Legendre-non-constant}. The relevant parametrization this time is 
\begin{equation}
\label{parametrization-non-constant}
z=Z(\chi(\zi)), \quad
\phi= \Phi(\zi,\chi(\zi)),
\end{equation}
where
$$
Z(\chi)=\frac{e^{s(\chi)} f(\chi)+\ka(\chi)\,e^{\ol{s(\chi)}} \ol{f(\chi)}}{1-|\ka(\chi)|^2} \ \mbox{ for } \ \chi\in\CC.
$$
For the details on how to construct $\Phi$, we refer the reader to Corollary \ref{cor:recovering-phi}.
\par
The functions $\chi(\zi)$, $s(\chi)$, and $\ka(\chi)$ are related to the derivatives of the function $\ell=\log n$, as explained in Section \ref{sec:Legendre-non-constant}. In particular, $\chi(\zi)$ is a homeomorphism associated to a metric induced by the derivatives of $\ell$, i.e. it is a univalent solution of the {\it Beltrami equation:} 
$$
\chi_{\ol{\zi}}=\si\,\chi_\zi.
$$ 
Here, $\si$ is a complex-valued function that depends on the derivatives of $\ell$. 
\par
In order to obtain the parametrization \ref{parametrization-non-constant}, it was crucial to use the elegant and now classical theory of \textit{pseudoanalytic functions} developed by L. Bers and I.~N.~Vekua (see \cite{Be}, \cite{Ve}, and \cite{AM} for another application).
If $\ell$ is constant, it turns out that $\si\equiv 0$, so that $\chi(\zi)$ can be taken to be the identity $\chi(\zi)=\zi$. Accordingly, one can compute that $s\equiv 0$ and $\ka(\chi)=\chi^2$, so that \eqref{parametrization-non-constant} reduces to \eqref{parametrization-constant}.
\par
In Section \ref{sec:light-shadow},
we analyse \eqref{parametrization-constant} and \eqref{parametrization-non-constant} and use them to give geometrical information about the set of critical points of $v=\im (\phi)$, and accordingly of caustics and the related light and shadow regions. This analysis goes along with and extends that developed (in the case of constant index of refraction) by the author and his co-authors in the series of papers \cite{MT1}-\cite{ MT5}, and \cite{CM}. In fact, in that case, it confirms that the set of critical points of $v=\im (\phi)$ is always a \textit{continuum}, made either of a finite number of segments or even of regions with non-empty interior (which correspond to the light regions). When the index of refraction is non-constant, the segments and light regions are distorted by the mapping $z\mapsto\chi(\zi(z,\ol{z}))$, which takes into account the effect of the index of refraction $n$.
\par
In our presentation, regularity of solutions is not an issue. The functions we consider should be intended as sufficiently regular as to make our computations work. Thus, in our statements we shall not specify regularity assumptions. Also, the parametrizations we obtain should also be intended to have a local character.

\section{Explicit solutions by Legendre transformation: \\
constant index of refraction}
 \label{sec:Legendre-constant}

\par
In this section, we will consider equation \eqref{eikonal-equation} in the plane and with a homogeneous index of refraction $n$ that we normalize to 1.

Introducing the complex variable $z=x+i\,y$ and its conjugate $\ol{z}=x-i\,y$  gives that \eqref{eikonal-equation} is equivalent to the following equation
\begin{equation}
\label{eq 2.1}
4\,\phi_z \phi_{\ol{z}} = 1,
\end{equation}
since $\pa_{x} = \pa_z + \pa_{\ol{z}}$ and $\pa_{y} = i\,(\pa_z - \pa_{\ol{z}})$.

Equation \eqref{eq 2.1} tells us that the point $(\phi_z,phi_{\ol{z}})$ belongs to a hyperbola, that we choose to parameterize by
\begin{equation}\label{eq 2.2}
\phi_z = \frac{1}{2\zi},\quad \phi_{\ol{z}} = \frac{1}{2} \zi,
\end{equation}
for some non-constant parameter $\zi=\zi(z,\ol{z})$.

A sufficiently smooth function $\phi$ satisfies Schwarz's lemma on second derivatives ($\phi_{z \ol{z}}=\phi_{\ol{z} z}$), hence \eqref{eq 2.2} gives that
\begin{equation}
\label{schwarz-equation}
\zi^2 \zi_z + \zi_{\ol{z}} = 0.
\end{equation}

We now change variables --- in fact operating a {\it Legendre transformation} --- and invert the roles of $z,\ol{z}$ and $\zi,\ol{\zi}$. Since it holds that
\begin{equation*}
\begin{array}{l}
\zi_z z_\zi + \zi_{\ol{z}} \ol{z}_\zi = 1, \\
\zi_z z_{\ol{\zi}} + \zi_{\ol{z}} \ol{z}_{\ol{\zi}} = 0,
\end{array}
\end{equation*}
we have that
\begin{equation}
\label{legendre}
\begin{array}{l}
\displaystyle \zi_z = \frac{\ol{z}_{\ol{\zi}}}{|z_\zi|^2-|z_{\ol{\zi}}|^2},  \quad \zi_{\ol{z}} = -\frac{z_{\ol{\zi}}}{|z_\zi|^2-|z_{\ol{\zi}}|^2}, \\
\displaystyle z_\zi = \frac{\ol{\zi}_{\ol{z}}}{|\zi_z|^2-|\zi_{\ol{z}}|^2},  \quad z_{\ol{\zi}} = -\frac{\zi_{\ol{z}}}{|\zi_z|^2-|\zi_{\ol{z}}|^2}. 
\end{array}
\end{equation}
\par
Notice that \eqref{schwarz-equation} gives that
$$
|\zi_z|^2-|\zi_{\ol{z}}|^2=|\zi_z|^2 (1-|\zi|^4)=|\zi_z|^2 (1+|\zi|^2)| (1-|\zi|^2).
$$
Since $\zi_z\ne 0$ (otherwise $\zi$ would be constant, thanks to \eqref{schwarz-equation}), we have that the jacobian determinant $|\zi_z|^2-|\zi_{\ol{z}}|^2$ is zero if and only if $|\zi|^2=1$. This holds if and only if $\na v=0$. In fact, we compute that 
\begin{equation*}
\label{singularity}
\zi-\frac1{\ol{\zi}}=2\,\phi_{\ol{z}}-2\,\ol{\phi}_{\ol{z}}=2\,i\,(v_x+i\,v_y),
\end{equation*}
so that
\begin{equation}
\label{jacobian-v}
\ol{\zi}\,(|\zi_z|^2-|\zi_{\ol{z}}|^2)=-2\,i\,|\zi_z|^2 (|\zi|^2+1)(v_x+i\,v_y).
\end{equation}
\par
Thus, away from the critical points of $v$, $\zi=\zi(z,\ol{z})$ is locally invertible by a function $z=z(\zi,\ol{\zi})$, and
\eqref{schwarz-equation} and \eqref{legendre} give:
\begin{equation}
\label{inverse-schwarz}
\zi^2 \ol{z}_{\ol{\zi}} - z_{\ol{\zi}} = 0.
\end{equation}
(For global invertibility, see \cite{Km}, for instance.)
Since \eqref{inverse-schwarz} can be rewritten as
\begin{equation*}
\pa_{\ol{\zi}} ( z - \zi^2 \ol{z}) = 0,
\end{equation*}
we infer that $z-\zi^2 \ol{z}$ must be an analytic function $f(\zi)$: i.e. we have that
\begin{equation}
\label{complex-rays}
z-\zi^2 \ol{z} = f(\zi),
\end{equation}
locally.

\begin{ex}
{\rm Let $f(\zi) = f_0 + f_1 \zi + f_2 \zi^2$ --- a second degree polynomial in $\zi$. Then, we obtain from \eqref{complex-rays} that
\begin{equation*}
(f_2+\ol{z}) \zi^2 + f_1 \zi + f_0 -z =0,
\end{equation*}
and hence 
\begin{equation*}
\zi = \frac{-f_1 + \sqrt{f_1^2 + 4(z-f_0)(\ol{z}+f_2)}}{2(\ol{z} +f_2)},
\end{equation*}
where the square root is the multi-valued complex-valued function.   
Plugging this expression of $\zi$ into \eqref{eq 2.2} and integrating gives that
$$
\phi=\frac12\sqrt{f_1^2+4(z-f_0)(\ol{z}+f_2)}-\frac12\, f_1\, \log\frac{\sqrt{f_1^2+4(z-f_0)(\ol{z}+f_2)}+f_1}{z-f_0}+c,
$$
where $c$ is an arbitrary constant.
Notice that, if we choose $f(\zi)=-1-\zi^2$, we recover \eqref{dist}.
}
\end{ex}

In general, we cannot expect to always obtain solutions in 
the explicit form as in this example. However, we will next show that
solutions in parametric form can always be constructed by this method. 

\begin{thm}
\label{th:parametrization}
Let $\Om\subset\CC$ be a simply connected open domain. Suppose that $\na v\ne 0$ in $\Om$. Then, $\phi=u+i\,v$ is a solution in $\Om$ of \eqref{eikonal-equation} if and only if there exists an analytic function $f(\zi)$ such that
\begin{equation}
\label{parametrization}
\begin{array}{l}
\ds z=\frac{f(\zi) + \zi^2 \overline{f(\zi)}}{1-|\zi|^4},  \\
\ds \phi= \frac{1}{1-|\zi|^4} \Big\{ \zi \overline{f(\zi)} + \frac{f(\zi)}{\zi} \Big\} - \frac{f(\zi)}{2\zi} + \frac{1}{2} \int \frac{f(\zi)}{\zi^2} d\zi. 
\end{array}
\end{equation}
\end{thm}

\begin{proof}
In fact, let us go back to \eqref{complex-rays}. By conjugating, we obtain the linear system in $z$ and $\ol{z},$
\begin{eqnarray*}
 &&z-\zi^2\ol{z} = f(\zi), \\
&&\ol{\zi}^2z-\ol{z} = -\overline{f(\zi)},
\end{eqnarray*}
that is solved by the $z$ given in \eqref{parametrization} and its conjugate.

Now, since $\phi_\zi = \phi_z z_\zi + \phi_{\ol{z}}\ol{z}_\zi$ and $\phi_{\ol{\zi}} = \phi_z z_{\ol{\zi}} + \phi_{\ol{z}}\ol{z}_{\ol{\zi}}$, from \eqref{eq 2.2} and \eqref{parametrization} we obtain  that
\begin{eqnarray*}
&& \phi_\zi = \frac{1+|\zi|^4}{(1-|\zi|^4)^2} \Big\{ \frac{f'(\zi)}{2 \zi} (1-|\zi|^4) + \ol{\zi}^2 f(\zi) + \overline{f(\zi)} \Big\}, \\
&& \phi_{\ol{\zi}} = \frac{2|\zi|^2}{(1-|\zi|^4)^2} \Big\{ \frac{\overline{f'(\zi)}}{2 \ol{\zi}} (1-|\zi|^4) + \zi^2 \overline{f(\zi)} + f(\zi) \Big\}.
\end{eqnarray*}
These equations can be locally integrated to get the second equation in \eqref{parametrization}. 
\par
Viceversa, thanks to \eqref{legendre} we have that
$$
\phi_z=\phi_\zi\,\zi_z+\phi_{\ol{\zi}}\,\ol{\zi}_z=\frac{\phi_\zi\,\ol{z}_{\ol{\zi}}-\phi_{\ol{\zi}}\ol{z}_{\zi}}{|z_\zi|^2-|z_{\ol{\zi}}|^2}, \quad
\phi_{\ol{z}}=\phi_\zi\,\zi_{\ol{z}}+\phi_{\ol{\zi}}\,\ol{\zi}_{\ol{z}}=\frac{-\phi_\zi\,z_{\ol{\zi}}+\phi_{\ol{\zi}}z_{\zi}}{|z_\zi|^2-|z_{\ol{\zi}}|^2}.
$$
Thus, we compute $\phi_\zi, \phi_{\ol{\zi}}, z_\zi$, and $z_{\ol{\zi}}$ from \eqref{parametrization} and discover that $4\,\phi_z\,\phi_{\ol{z}}=1$, after tedious calculations.
\end{proof}

\smallskip

\section{Explicit solutions by Legendre transformation: \\
non-constant index of refraction}
 \label{sec:Legendre-non-constant}

If the index of refraction in \eqref{eikonal-equation} is non-constant,  then solutions of \eqref{eikonal-equation} still admit a parametrization in terms of some analytic function. In this section, we shall show how to construct the necessary modifications.

\begin{rem}
{\rm
There is a special case in which one can easily contruct solutions, when the index of refraction in \eqref{eikonal-equation} is non-constant. In fact, let $w:\Om\to w(\Om)$ be a bi-holomorphic homeomorphism and suppose that $n$ has the special form:
$$
n(z)=|w'(z)| \ \mbox{ for } \ z\in\Om.
$$
Let $\phi$ be a solution of \eqref{eikonal-equation} in $\Om$. Notice that the function $\psi$, defined in $w(\Om)$ by $\psi(w)=\phi(z)$ for $z\in\Om$, is such that
$$
\phi_z=\psi_w\, w'(z), \quad \phi_{\ol{z}}=\psi_{\ol{w}}\, \ol{w'(z)}.
$$
Thus, we have that
$$
|w'(z)|^2=4\,\phi_z\,\phi_{\ol{z}}=4\,\psi_w\,\psi_{\ol{w}}\,|w'(z)|^2.
$$
Since $w'(z)\ne 0$, we infer that  $4\,\psi_w\,\psi_{\ol{w}}=1$ in $w(\Om)$. Therefore, in this case, we can simply reduce our problem to the case of constant index of refraction.
\par
In general, however, $n(z)$ may not be the modulus of an analytic function. In what follows, we shall describe how to proceed. In Section \ref{sec:light-shadow}, we shall see that, even if the procedure can hardly be used to construct explicit solutions, it still provides useful geometric information about caustics, shadow regions, and light regions.   
}
\end{rem}

We start with the following simple proposition.

\begin{prop}
Let $N=2$ and suppose that $n$ is of class $C^1$ and bounded away from zero.  Set $\ell=\log n$.  
\par
If $\phi$ is a solution of \eqref{eikonal-equation}, then the function $\zi$ defined by
$$
\zi=\frac{2}{n}\,\phi_{\ol{z}}
$$
is a solution of the equation:
\begin{equation}
\label{a-b-c-d}
a\,\zi_z+b\,\ol{\zi}_{\ol{z}}+c\,\zi_{\ol{z}}+d\,\ol{\zi}_z=0,
\end{equation}
where
\begin{equation}
\label{coeff-a-b-c-d}
a=1+\zi\,\ell_\zi, \quad b=-\frac1{\zi}\,\ell_{\ol{\zi}}, \quad c=\frac1{\zi^2}\,(1-\zi\,\ell_\zi).
\quad d=\zi\,\ell_{\ol{\zi}}.
\end{equation}
\end{prop}

\begin{proof}
Notice that  \eqref{eikonal-equation} becomes:
$$
4\,\phi_{z}\,\phi_{\ol{z}}=n^2=e^{2\ell}.
$$
Next, we have that
$$
\phi_{z}=\frac{e^\ell}{2\,\zi}, \quad \phi_{\ol{z}}=\frac{e^\ell}{2}\, \zi.
$$
Thus, by Schwartz's theorem on mixed second derivatives, we infer that
$$
\frac1{\zi}\,\ell_{\ol{z}}-\frac1{\zi^2}\,\zi_{\ol{z}}=\zi\,\ell_z+\zi_z.
$$
Therefore, we get:
$$
\zi_z+\frac1{\zi^2}\,\zi_{\ol{z}}=-\zi\,\ell_z+\frac1{\zi}\,\ell_{\ol{z}}=
-\zi\,(\ell_\zi\,\zi_z+\ell_{\ol{\zi}}\,\ol{\zi}_z)+\frac1{\zi}\,(\ell_\zi\,\zi_{\ol{z}}+\ell_{\ol{\zi}}\,\ol{\zi}_{\ol{z}}).
$$
We can rearrange this equation and obtain \eqref{a-b-c-d}, when \eqref{coeff-a-b-c-d} is in force.
\end{proof}

\begin{rem}
{\rm
Similarly to what observed in Section \ref{sec:Legendre-constant}, we have that the jacobian $|\zi_z|^2-|\zi_{\ol{z}}|^2\ne 0$ if and only if $\na v\ne 0$. In fact, in this case, \eqref{jacobian-v} changes into
\begin{equation*}
\label{jacobian-v-n}
n\,\ol{\zi}\,(|\zi_z|^2-|\zi_{\ol{z}}|^2)=-2\,i\,|\zi_z|^2 (|\zi|^2+1)(v_x+i\,v_y).
\end{equation*}
%Also, by conjugation, from \eqref{a-b-c-d} we infer that
%\begin{eqnarray*}
%&&a\,\zi_z+b\,\ol{\zi}_{\ol{z}}+c\,\zi_{\ol{z}}+d\,\ol{\zi}_z=0, \\
%&&\ol{b}\,\zi_z+\ol{a}\,\ol{\zi}_{\ol{z}}+\ol{d}\,\zi_{\ol{z}}+\ol{c}\,\ol{\zi}_z=0.
%\end{eqnarray*}
%Once $|a|^2-|b|^2\ne 0$, we compute that
%$$
%\zi_z=A\,\zi_{\ol{z}}+B\,\ol{\zi}_{z},
%$$
%where
%$$
%A=\frac{b\,\ol{d}-\ol{a}\,c}{|a|^2-|b|^2}, \quad B=\frac{b\,\ol{c}-\ol{a}\,d}{|a|^2-|b|^2}.
%$$
%Hence, we have that
%\marginpar{sviluppare}
%$$
%|\zi_z|^2-|\zi_{\ol{z}}|^2=(|A|^2+|B|^2-1)\,|\zi_{\ol{z}}|^2+2\,\re[A\,\ol{B}\,\zi_{\ol{z}}^2].
%$$
%$$
% \ge
%[(|A|-|B|)^2-1]\,|\zi_{\ol{z}}|^2
%$$
Since $\zi$ is non-constant, we can assume that $\zi_{\ol{z}}\ne 0$.}
\end{rem}

\begin{prop}
Suppose that $\zi$ is a solution of \eqref{a-b-c-d} such that $|\zi_{\ol{z}}|^2-|\zi_z|^2\ne 0$.
Then, the function $z(\zi)$ such that $z(\zi(z))=z$ and $\zi(z(\zi))=\zi$ satisfies the equation:
\begin{equation}
\label{b-a-c-d}
b\,z_\zi+a\,\ol{z}_{\ol{\zi}}-c\,z_{\ol{\zi}}-d\,\ol{z}_\zi=0.
\end{equation}
\end{prop}

\begin{proof}
In order to obtain \eqref{b-a-c-d}, it is sufficient to apply the first two equations in \eqref{legendre} to \eqref{a-b-c-d}. 
\end{proof}

The following result is the crucial step for the construction of parametrized solutions of \eqref{eikonal-equation} with non-constant index of refraction.

\begin{thm}
\label{th:parametrization-z-n}
Let $a, b, c$, and $d$ be given by \eqref{coeff-a-b-c-d}, set
\begin{equation}
\label{A-coefficients}
\begin{array}{ll}
\displaystyle  A_{11}=2\,\frac{\im\bigl[(a+b)\,(\ol{c+d})\bigr]}{|a+c|^2-|b+d|^2}, \quad
&\displaystyle A_{12}=\frac{|a+d|^2-|b+c|^2}{|a+c|^2-|b+d|^2}, \\
\displaystyle  A_{21}=\frac{|a-d|^2-|b-c|^2}{|a+c|^2-|b+d|^2}, \quad 
&\displaystyle A_{22}=2\,\frac{\im\bigl[(\ol{a-b})\,(c-d)\bigr]}{|a+c|^2-|b+d|^2},
\end{array}
\end{equation}
and suppose that 
\begin{equation}
\label{ellipticity}
4\,A_{12}\,A_{21}-(A_{11}+A_{22})^2>0, \ A_{12}>0.
\end{equation}
\par
Let $z$ be a solution of \eqref{b-a-c-d}. Then there exist a quasi-conformal homeomorphism $\chi=\chi(\zi)$, an analytic function $f(\chi)$, and two complex-valued functions  $\ka(\chi)$ and  $s(\chi)$ such that
\begin{equation}
\label{parametrization-z-non-constant}
z(\zi)=Z(\chi(\zi)),
\end{equation}
with 
\begin{equation}
\label{parametrization-z-non-constant-F}
Z(\chi)=\frac{e^{s(\chi)} f(\chi)+\ka(\chi)\,\ol{f(\chi)\,e^{s(\chi)}}}{1-|\ka(\chi)|^2}.
\end{equation}
\end{thm}
\begin{proof}
Let $z=x+i\,y$ and $\zi=\xi+i\,\eta$. Lengthy computations give that \eqref{b-a-c-d} is equivalent to the system:
\begin{equation}
\label{system}
y_\xi=A_{11}\,x_\xi+A_{12}\,x_\eta, \quad
-y_\eta=A_{21}\,x_\xi+A_{22}\,x_\eta.
\end{equation}
The coefficients of the system are given by \eqref{A-coefficients}.
\par
Next, we follow the arguments of \cite[Section 12]{Be} on pseudo-analytic functions. Since \eqref{ellipticity} holds,
the system \eqref{system} is declared elliptic and $z(\zi)$ satisfies the equation:
\begin{equation}
\label{pseudo-analytic}
z_{\ol{\zi}}=\mu\,z_\zi+\nu\,\ol{z}_{\ol{\zi}},
\end{equation}
where
\begin{eqnarray*}
&&\mu=\frac{A_{11}+A_{22}+i\,(A_{12}-A_{21})}{2-(A_{12}+A_{21})+A_{12}A_{21}-A_{11}A_{22}}, \\ 
\\
&&\nu=\frac{1+A_{11}A_{22}-A_{12}A_{21}-i\,(A_{22}-A_{11})}{2-(A_{12}+A_{21})+A_{12}A_{21}-A_{11}A_{22}}.
\end{eqnarray*}
It turns out that $|\mu|+|\nu|<1$, under the assumption \eqref{ellipticity}.
\par
Equation \eqref{pseudo-analytic} can be conveniently transformed into a canonical form.
In fact, since $|\mu|+|\nu|<1$, it exists a homeomorphism $\chi=\chi(\zi)$
with respect to the metric associated to the system \eqref{system}, i.e. a univalent solution of the \textit{Beltrami equation}
\begin{equation}
\label{beltrami}
\chi_{\ol{\zi}}=\si\,\chi_\zi,
\end{equation}
such that the function $Z$ defined by $z(\zi)=Z(\chi(\zi))$ is a solution of the canonical equation
\begin{equation}
\label{canonical}
Z_{\ol{\chi}}=\ka\,\ol{Z}_{\ol{\chi}}.
\end{equation}
\par
In terms of the coefficients in the system \eqref{system}, we get:
\begin{equation}
\label{def-sigma}
\si=\frac{A_{12}-A_{21}-i\,(A_{11}+A_{22})}{A_{12}+A_{21}+\sqrt{4\,A_{12}\,A_{21}-(A_{11}+A_{22})^2}}.
\end{equation}
Under the assumption \eqref{ellipticity}, it holds that $|\si|<1$. Also, we compute that
\begin{equation}
\label{def-sigma-hat}
\ka=\frac{\nu}{1-\mu\,\ol{\si}},
\end{equation}
and again $|\ka|<1$.
\par
We can further transform the equation for $Z$ by setting $W=Z-\ka\,\ol{Z}$. It turns out that $W$ is a solution of the equation:
$$
W_{\ol{\chi}}=B\,W+C\,\ol{W},
$$
where
$$
B=-\frac{\ol{\ka}\,\ka_{\ol{\chi}}}{1-|\ka|^2}, \quad C=-\frac{\ka_{\ol{\chi}}}{1-|\ka|^2}.
$$
This means that $W$ is a pseudo-analytic function of the first kind in the $\chi$-plane.
\par
It is well known that pseudo-analytic functions of the first kind  satisfy the similarity principle (see \cite[Section 10]{Be}). Thus, we can infer that there exist an analytic function $f$ and a complex-valued function $s$ such that
$$
W(\chi)=e^{s(\chi)} f(\chi).
$$
\par
This last formula and the fact that
$$
Z=\frac{W+\ka\,\ol{W}}{1-|\ka|^2}
$$ 
give the desired conclusion.
\end{proof}

\begin{cor}
\label{cor:recovering-phi}
Under the assumptions of Theorem \ref{th:parametrization-z-n}, any solution $\phi$ 
of \eqref{eikonal-equation} can be parametrized as in \eqref{parametrization-non-constant},
$$
z=Z(\chi(\zi)), \quad \phi=\Phi(\zi, \chi(\zi)),
$$
where
\begin{eqnarray*}
&&\phi_\zi = \frac{N(\chi)}{2\,\zi}\,\Bigl[(1+\zi^2 \ol{\ka})\, Z_\chi\,\chi_\zi+
\ol{\si}\,(\ka+\zi^2)\,\ol{Z_\chi \chi_\zi}\Bigr], \\
&&\phi_{\ol{\zi}}=\frac{N(\chi)}{2\,\zi}\,\Bigl[\si\,(1+\zi^2 \ol{\ka})\, Z_\chi\,\chi_\zi+
(\ka+\zi^2)\,\ol{Z_\chi \chi_\zi}\Bigr],
\end{eqnarray*}
where $N=n\circ Z$.
\end{cor}

\begin{proof}
We proceed as in the proof of Theorem \ref{th:parametrization}. In fact, we still observe that 
$\phi_\zi = \phi_z z_\zi + \phi_{\ol{z}}\ol{z}_\zi$ and $\phi_{\ol{\zi}} = \phi_z z_{\ol{\zi}} + \phi_{\ol{z}}\ol{z}_{\ol{\zi}}$, and hence we obtain that
$$
\phi_\zi = \frac{n}{2\,\zi}\,z_\zi+ \frac{n\,\zi}{2}\,\ol{z}_\zi, \quad 
\phi_{\ol{\zi}} = \frac{n}{2\,\zi}\,z_{\ol{\zi}}+ \frac{n\,\zi}{2}\,\ol{z}_{\ol{\zi}}.
$$
Thus, we find the desired two formulas for $\phi_\zi$ and $\phi_{\zi}$, by using \eqref{beltrami} and \eqref{canonical}.
\par
By integrating the obtained two formulas one for $\phi_\zi$ and $\phi_{\zi}$ we can finally recover $\phi=\Phi(\zi, \chi(\zi))$.
\end{proof}

\section{On caustics and light and shadow regions} 
\label{sec:light-shadow}
As already observed in the previous sections, real-valued solutions of \eqref{eikonal-equation} describe the propagation of light along rays. Thus, the region swept by rays --- the light region --- corresponds to the set in which the eikonal $\phi=u+i\,v$ is real-valued or, up to some suitable normalization, to the set of critical points of $v$. The envelope of rays create a so-called caustic that bounds the light region.  Outside that region, i.e. beyond the caustic,  we have that $\na v\ne 0$ and, up to normalizing an additive imaginary constant, we can assume that $v<0$, so that the exponentials in \eqref{sviluppo WKBJ} or \eqref{sviluppo EWT} rapidly decay to zero as the wave number $k$ becomes large. This is the so-called shadow region.
\par
In this section, we shall thus investigate on the region in which $\na v\ne 0$, by means of our parametrization. 

\subsection{The case of constant index of refraction $n$}
As noticed in Section \ref{sec:Legendre-constant}, \eqref{jacobian-v} informs us that the map $z\mapsto\zi(z)$ is locally invertible, and hence \eqref{inverse-schwarz} holds as well.
\par
An inspection of \eqref{parametrization} tells us that the parametrization is not defined on the unit disk on which $|\zi|^2=1$. By \eqref{jacobian-v}, we see that this holds exactly if $\na v=0$. Next, observe that equation \eqref{complex-rays} defines a complex ray for each fixed $\zi$. Since $\zi=\xi+i\,\eta$, \eqref{complex-rays} can be re-written as
\begin{eqnarray*}
&&(1-\xi^2+\eta^2)\,x-2\,\xi\,\eta\, y=\re[f(\xi+i\,\eta)], \\ 
&&-2\,\xi\,\eta\, x+(1+\xi^2-\eta^2)\,y=\im[f(\xi+i\,\eta)].
\end{eqnarray*}
If we use polar coordinates $r\,e^{i\,\te}$ for $\zi$, the system becomes:
\begin{eqnarray*}
&&[1-r^2\cos(2\te)]\,x-r^2\sin(2\,\te)\, y=\re[f(r\,e^{i\,\te})],  \\ 
&&-r^2\sin(2\,\te)\, x+[1+r^2\cos(2\te)]\,y=\im[f(r\,e^{i\,\te})].
\end{eqnarray*}
It is easy to compute that the determinant of this linear system equals $1-r^4$. Thus, when $\na v\ne 0$, we have that $|\zi|^2\ne 1$, and hence each point $\zi=r\,e^{i\,\te}$ determines a unique point $z=x+i\,y$. In other words, the parametrization uniquely detemines all the regular points of $v$, which correspond to the points in the shadow region.
\par
In the light region, i.e. when $\na v=0$, 
the system degenerates into the two equations: 
\begin{equation}
\label{coincident-lines}
\begin{array}{l}
[1-\cos(2\te)]\,x-\sin(2\,\te)\, y=\re[f(e^{i\,\te})],  \\ 
-\sin(2\,\te)\, x+[1+\cos(2\te)]\,y=\im[f(e^{i\,\te})].
\end{array}
\end{equation}
Since in this case the determinant of this system is always zero, the system has solutions if and only if
\begin{equation}
\label{condition}
\cos\te\,\re[f(e^{i\,\te})]+\sin\te\,\im[f(e^{i\,\te})]=0
\end{equation}
or, in complex notation if and only if
\begin{equation}
\label{condition-complex}
\re\bigl[f(e^{i\,\te})\, e^{-i\,\te}]=0.
\end{equation}
\par
If this condition holds for some $\te$, then the equation $\na v=0$ is satisfied at all points of the line defined by either one equation in \eqref{coincident-lines}. If \eqref{condition} does not hold for some $\te$, the lines defined in \eqref{coincident-lines} are parallel, and hence no points on them is a critical point of $v$.
We can summarize this analysis into the following result.

\begin{prop} 
\label{prop:rays}
It holds that the parametrization \eqref{parametrization} maps:
\begin{enumerate}[(i)]
\item
each point $\zi$ outside the unit circle onto a point of the shadow region;
\item
each point $\zi$ on the unit circle that satisfies $\re\bigl[f(e^{i\,\te})\,e^{-i\,\te}]=0$ onto a segment  within the light region, whose equation is given in \eqref{coincident-lines} (either one will do);
\item
any other point on the unit circle to $\infty$.
\end{enumerate}
\end{prop}

\begin{rem}
{\rm
Let $\phi=u+i\,v$ be a solution of \eqref{eikonal-equation} in an open domain $\Om\subseteq\CC$. Proposition \ref{prop:rays} and \eqref{jacobian-v} inform us that the set of critical points of $v$,
$$
\cC_v=\{ z\in\Om: \na v(z)=0\},
$$
does not contain isolated points. In fact, $\cC_v$ is mapped by the parametrization for $z$ in \eqref{parametrization} from the set $\cS_\phi$ defined by
$$
\cS_\phi=\{ \zi=e^{i\,\te}: \re[f(e^{i\,\te})\,e^{-i\,\te}]=0\}.
$$
Thus, discrete points in $\cS_\phi$ are mapped to segments in $\cC_v$ (as it happens, for instance, for \eqref{dist}). If $\cS_\phi$ contains an arc, then a corresponding pencil of segments sweeps a subset of $\cC_v$ with non-empty interior. This analysis confirms similar results previously obtained in \cite{MT1, CM}, by other means and different viewpoints.
\par
Notice that, because of the condition \eqref{condition-complex}, the parametrization for $z$ in \eqref{parametrization} does converge as $\zi$ converges to points in $\cS_\phi$. In fact, in the limit as $|\zi|\to 1$, we obtain from \eqref{parametrization} that $z(\zi)$ is a point in the shadow region that converges to the point:
$$
z=-\frac14\,e^{i\,\te} [f'(e^{i\,\te})+\ol{f'(e^{i\,\te})}]-\frac12\,e^{2\,i\,\te} \ol{f(e^{i\,\te})}.
$$ 
When $\cS_\phi$ contains an arc, by using \eqref{condition-complex} and its derivative, we then can infer that
\begin{equation*}
\label{caustic-parametrization}
z=f(e^{i\,\te})-\frac12\,e^{i\,\te} f'(e^{i\,\te}),
\end{equation*}
for $e^{i\,\te}$ belonging to that arc. The last equation provides a parametrization of a caustic, which separates the shadow region from the light region $\cC_v$.
\par
By means of \eqref{parametrization}, we can also compute the values of $\phi$ on $\cC_v$.
}
\end{rem}

\begin{ex}
{\rm
One can concretely construct an analytic function $f(\zi)$ in a neighborhood of the unit circle such that  $\cS_v$ contains an arc. It is sufficient to construct $\re[f(\zi)/\zi]$ as a harmonic function which is zero on an arc of the unit circle, and then extend it analytically on a neighborhood.
\par
In fact, by the classical Poisson's formula (see \cite{Ru}), 
$$
g(\zi)=\frac1{2\pi}\int_{-\pi}^\pi \frac{e^{it}+\zi}{e^{i t}-\zi}\,\re g(e^{i t})\,dt
$$
represents the analytic function in the unit disk with real part assigned on the unit circle.
We thus can choose, for instance, that
$$
\re g(e^{i t})=\re\bigl[f(e^{i t})/e^{i t}\bigr]=\max(|t|-\tau,0),
$$
with $0<\tau<\pi$, to obtain the formula:
$$
f(\zi)=\frac1{\pi}\int_0^{\pi-\tau}\frac{\zi\,(1-\zi^2)\,t\,dt}{1+\zi^2-2\,\zi\,\cos(\si+t)}.
$$
This formula gives an analytic function outside the arc $\ga=\{e^{i \te}:\tau\le |\te|\le\pi\}$, continuous up to the boundary, and such that $\re[f(e^{i \te})/e^{i \te}]=0$ for $0\le |\te|<\tau$.
This choice of $f(\zi)$ will thus generate a caustic, thanks to the parametrization \eqref{parametrization-constant}, and a set $\cC_v$ with non-empty interior.
\par
Of course, the values of $\re[f(e^{i \te})/e^{i \te}]$ on $\ga$ can be arbitrarily changed, in order to obtain other examples.
}
\end{ex}

\subsection{The case of non-constant index of refraction}
The situation is more complicated, and slightly different since we do not directly work with $\zi(z)$, but rather with $z(\zi)$. However, the geometric information we derive is quite similar. \par
Once the coefficients in \eqref{A-coefficients} are determined in terms of $\ell$, we can compute $\si$ and $\ka$, thanks to \eqref{def-sigma} and \eqref{def-sigma-hat}.
After a homeomorphism $\chi(\zi)$ satisfying the Beltrami equation \eqref{beltrami} is determined, we can examine the parametrization \eqref{parametrization-z-non-constant}-\eqref{parametrization-z-non-constant-F}. 
\par
Recall that the functions $\zi(z)$ and $z(\zi)$ are the inverse of one another. Thus, from \eqref{parametrization-z-non-constant} we obtain that
$$
z=z(\zi(z))=Z(\chi(\zi(z))),
$$
i.e. the composition $\om=\chi\circ\zi$ takes points of the $z$-plane to points of the $Z$-plane, so that any curve in the $Z$-plane is deformed by the inverse $\om^{-1}$ onto a corresponding curve in the $z$-plane.
\par
Now, notice that \eqref{parametrization-z-non-constant} and \eqref{parametrization-z-non-constant-F} give that
\begin{equation}
\label{deformed-ray}
Z-\ka(\chi)\,\ol{Z}=e^{s(\chi)} f(\chi).
\end{equation}
Thus, any point in the $\chi$-plane determines a complex ray in the $Z$-plane, and this is deformed by $\om^{-1}$ to obtain in the $z$-plane a complex ray in a metric induced by the index of refraction $n$.
\par
An analysis similar to that of Proposition \ref{prop:rays} informs us that the points $\chi$ not belonging to the set $\ka^{-1}(\SS^1)$ are mapped onto the shadow region. Each point $\chi$ satisfying
\begin{equation}
\label{caustic-conditions}
|\ka(\chi)|^2=1 \ \mbox{ and } \ e^{s(\chi)} f(\chi)+\ka(\chi)\,\ol{f(\chi)\,e^{s(\chi)}}=0
\end{equation}
is mapped to a segment with equation \eqref{deformed-ray} in the $Z$-plane, and hence to a deformed segment in the $z$-plane, by means of the mapping $\om^{-1}$. When \eqref{caustic-conditions} is satisfied  on a curve in the $\chi$-plane, these deformed segments sweep the light region and their envelope creates a caustic separating the light from the shadow region.

\section*{Acknowledgements}
The author wish to thank Giulio Ciraolo for some useful discussions. The paper is partially supported by the Gruppo Nazionale per l'Analisi Matematica, la Probabilit\`a e le loro Applicazioni (GNAMPA) dell'Istituto Nazionale di Alta Matematica (INdAM).

\end{document}